\documentclass[12pt,reqno]{amsart}
\usepackage{amsmath,amsthm}
\usepackage{amssymb}

\usepackage[margin=1.5in]{geometry}
\usepackage[all,cmtip]{xy}
\usepackage{microtype}

\numberwithin{equation}{section}

\newtheorem{Thm}[subsection]{Theorem}

\newtheorem{Lem}[subsection]{Lemma}

\newtheorem{Cor}[subsection]{Corollary}

\newtheorem{Exm}[subsection]{Example}

\newcommand{\br}{\mathbb{R}}

\newcommand{\bz}{\mathbb Z}

\newcommand{\wt}{\widetilde}
\begin{document}

\title[Stiefel-Whitney classes]{On Stiefel-Whitney Classes of vector bundles over real Stiefel Manifolds}

\author{Prateep Chakraborty}
\address{Stat-Math Unit, Indian Statistical Institute, 8th Mile, Mysore Road, RVCE Post, Bangalore 560059, INDIA.}
\email{chakraborty.prateep@gmail.com}

\author{ Ajay Singh Thakur}
\address{Stat-Math Unit, Indian Statistical Institute, 8th Mile, Mysore Road, RVCE Post, Bangalore 560059, INDIA.}

\email{thakur@isibang.ac.in}

\thanks{The research of first author is supported by NBHM postdoctoral fellowship.
The research of second author is supported by DST-Inspire Faculty Scheme (IFA-13-MA-26)}



\begin{abstract} In this article we show that there are at most two integers up to $2(n-k)$, which can occur as the degrees of nonzero Stiefel-Whitney classes of vector bundles over the Stiefel manifold $V_k(\br^n)$. In the case when $n> k(k+4)/4$, we show that if $w_{2^q}(\xi)$ is the first nonzero Stiefel-Whitney class of a vector bundle $\xi$ over $V_k(\br^n)$ then $w_t(\xi)$ is zero if $t$ is not a multiple of $2^q.$ In addition, we give relations among Stiefel-Whitney classes whose degrees are multiples of $2^q$.
\end{abstract}
\maketitle

\section{Introduction}The real Stiefel manifold $V_k(\br^n)$ is the set of all orthonormal $k$-frames in $\br^n$ and it can be identified with the homogeneous space $SO(n)/SO(n-k)$. The main aim of this article is to study Stiefel-Whitney classes of vector bundles over a real Stiefel manifold. 

Recall that the degree of the first nonzero Stiefel-Whitney class of a vector bundle over a CW-complex $X$ is a power of $2$ (cf., for example, \cite[page 94]{milnor}). In the case when $X$ is a $d$-dimensional Sphere $S^d$, it is a theorem of Atiyah-Hirzebruch \cite[Theorem 1]{atiyah} that $d$ can occur as the degree of a nonzero Stiefel-Whitney class of a vector bundle over $S^d$ if and only if $d=1,2,4,8.$ The possible Stiefel-Whitney classes of vector bundles over Dold manifold and stunted real projective space are completely determined by Stong \cite{stong} and Tanaka \cite{tanaka}, respectively. In this article we shall deal with case $X = V_k(\br^n)$ and derive certain results on Stiefel-Whitney classes. 

In \cite{knt}, it was observed that for a vector bundle $\xi$ over $V_k(\br^n)$, $n>k$, the Stiefel-Whitney class $w_{n-k}(\xi) = 0$ if $n-k \neq 1,2,4,8$ and $w_{n-k+1}(\xi) = 0$ if $n-k = 2,4,8$. We extend this observation to get the following theorem  where we show that there are at most two integers up to $2(n-k)$, which can occur as the degrees of nonzero Stiefel-Whitney classes of any vector bundle over $V_k(\br^n)$.
\begin{Thm}\label{tanaka}
 Let $\xi$ be a vector bundle over $V_k(\br^n), n> k$. Let $i$ be a positive integer with $i\leq 2(n-k)$.  Then $w_i(\xi)=0$ if one of the following conditions is satisfied.
 \begin{enumerate}
 \item $n-k \neq 1,2,4,8$ and   $i \neq 2^{\varphi(n-k-1)}$ 
 \item $n-k = 1,2,4,8$ and $i \neq n-k, 2(n-k)$.
 
 \end{enumerate}  \end{Thm}
In the above theorem, $\varphi(m)$,  for a non-negative integer $m$,  is the number of integers $l$ such that $0<l\leq m$ and $l\equiv 0,1,2,4 \pmod 8$.

From Theorem \ref{tanaka}, we observe that if $i$ is the first nonzero Stiefel-Whitney class of a vector bundle $\xi$ over $V_k(\br^n)$ and $i\leq2(n-k)$, then $i$ is of the  form $2^{\varphi(n-k-1)}.$ Now in the next theorem,  for a vector bundle over $V_k(\br^n)$, we derive  vanishing of certain  Stiefel-Whitney classes whose degrees depend on the degree of the first nonzero Stiefel-Whitney class.
\begin{Thm}\label{main}
	 Let $n > k(k+4)/4$. Let $\xi$ be a vector bundle over $V_k(\mathbb{R}^n)$ with first non-zero Stiefel Whitney class in degree $2^q$. If $i$ is a multiple of $2^q$ and is written as $i =  2^{q+t_1}+2^{q+t_2}+\cdots + 2^{q+t_m} \mbox{ with } t_j\geq0 \mbox{ and } t_j<t_{j+1}$, then $w_i(\xi) = w_{2^{q+t_1}}(\xi)\cdot w_{2^{q+t_2}}(\xi)\cdots w_{2^{q+t_m}}(\xi)$. Further, if $i$ is not a multiple of $2^q$, then $w_i(\xi) = 0$.
	 \end{Thm}
	 
	  Recall (\cite{naolekar}) that $\mbox{ucharrank}(X)$ of $X$ is the maximal degree up to which every cohomology class of $X$ is a polynomial in the Stiefel-Whitney classes of a vector bundle over $X$. The \textit{ucharrank} of $V_k(\br^n)$ was computed
	  in \cite{knt}, except for the cases $n-k = 4,8$, in which cases it was shown that $\mbox{ucharrank}(V_k(\br^n))$ is bounded above by $n-k$. In Example \ref{exmp}, we construct a vector bundle $\xi$ over $V_k(\br^n)$, when $n-k = 4,8$,  such that $w_{n-k}(\xi) \neq 0$ and hence improve the result in \cite{knt} to obtain $\mbox{ucharrank}(V_k(\br^n)) = n-k$.

	 To prove our results we need the Steenrod algebra action on the mod-2 cohomology ring $H^*(V_k(\br^n);\bz_2)$. 
	 Recall \cite[Proposition 9.1 and 10.3]{borel} that the cohomology ring $H^*(V_k(\mathbb R^n);\bz_2)$ has a simple system of generators
	 	$a_{n-k}, a_{n-k+1},\ldots , a_{n-1}$, where $a_i\in H^i(V_k(\mathbb R^n)$ with the following
relations:	 	$$a_i^2= \left \{\begin{array}{cl} a_{2i} & \mbox{ if } 2i \leq n-1 \\0 & \mbox{ otherwise. } \end{array}\right.$$
	 	The action of Steenrod algebra is completely determined by knowing that (see \cite{borel}, Remarque 2 in \S 10):
	 $$Sq^i(a_j)=\left\{\begin{array}{cl}
	 \displaystyle{j \choose i}a_{j+i} & \mbox{if $ j+i\leq n-1$,}\\
	 & \\
	 0 & \mbox{otherwise.}
	 \end{array}\right.$$
	 
\subsection*{Notations} In this article we shall only consider real Stiefel manifold. The cohomology ring will always be with $\bz_2$-coefficients, unless specified otherwise.

%

\section{Proof of Theorem \ref{tanaka}}

We first recall the description of Stiefel-Whitney classes of vector bundles over stunted real projective space, due to Tanaka \cite{tanaka}. For $n>k$, let $P_{n,k}$ be the stunted real projective space obtained from  $\mathbb{RP}^{n-1}$ by collapsing the subspace $\mathbb{RP}^{n-k-1}$ to a point. Consider the following cofibration sequence
$$\mathbb{RP}^{n-k-1}\longrightarrow\mathbb{RP}^{n-1}\stackrel{g}\longrightarrow P_{n,k}.$$
The induced map in cohomology $g^*:H^j(P_{n,k})\rightarrow H^j(\mathbb{RP}^{n-1})$ is an isomorphism when $n-k\leq j\leq n-1$. Therefore, for any vector bundle $\xi$ over $P_{n,k}$, the Stiefel-Whitney class $w_j(\xi)\neq0$ if and only if $w_j(g^*(\xi))\neq0$. From \cite{adams} (also cf. \cite{tanaka}), we know that the image $g^*:\wt{KO}(P_{n,k})\rightarrow \wt{KO}(\mathbb{RP}^{n-1})$ is generated by $2^{\varphi(n-k-1)}\gamma$, where $\gamma$ is the canonical line bundle over $\mathbb{RP}^{n-1}$ and for a non-negative integer $m$,  $\varphi(m)$ is as defined in the Introduction.  If we denote the generator of $H^*(\mathbb{RP}^{n-1})$ by $t$, then for any integer $d$, the total Stiefel-Whitney class of the element $d2^{\varphi(n-k-1)}\gamma$ in the image of $g^*$, is given as
$$w(d2^{\varphi(n-k-1)}\gamma)=(1+t)^{d2^{\varphi(n-k-1)}}=(1+t^{2^{\varphi(n-k-1)}})^d.$$ Therefore, the nonzero Stiefel-Whitney classes of any vector bundle $\xi$ over $P_{n,k}$ can occur only in degrees $r2^{\varphi(n-k-1)}$ for some integer $r$.
 
 To prove Theorem \ref{tanaka}, we shall use the following observation. For a non-negative integer $m$, we note that if $m\equiv 1,2,3,4,5\pmod 8$, then $\varphi(m)=[m/2]+1$ and if $m\equiv 0,6,7\pmod 8$, then $\varphi(m)=[m/2]$. From here we can conclude that for a positive integer $m$, we have $2^{\varphi(m-1)} \geq m$ and the equality holds only if $m = 1,2,4 \mbox{ and } 8$.

%

\begin{proof}[\textbf{Proof of Theorem \ref{tanaka}}]  
Recall (cf. \cite{james}) that there is a cellular embedding $f:P_{n,k}\hookrightarrow V_{k}(\br^n)$ such that the cellular pair $(V_{k}(\br^n), P_{n,k})$ is $2(n-k)$ connected (cf. \cite[Proposition 1.3]{james}).
	Hence,  the induced map in cohomology $f^*:H^j(V_{k}(\br^n)) \rightarrow H^j(P_{n,k})$ is injective for $j\leq 2(n-k)$. Therefore, for a vector 
bundle $\xi$ over $V_k(\br^n)$, the Stiefel-Whitney class $w_j(\xi)\neq0$ if and only if $w_j(f^*(\xi))\neq0$ when $n-k\leq j\leq 2(n-k).$ By the description of Stiefel-Whitney classes of vector bundles over $P_{n,k}$, as discussed above, it follows that $w_j(\xi)=0$ if $n-k\leq j\leq \min\{n-1,2(n-k)\}$ and $j\neq r2^{\varphi(n-k-1)}$ for any integer $r$. As, $2^{\varphi(n-k-1)} \geq (n-k)$ and the equality holds only if $n-k = 1,2,4$, and $8$, the only multiples of $2^{\varphi(n-k-1)}$ that can occur in between $(n-k)$ and $2(n-k)$ are $2^{\varphi(n-k-1)}, 2^{\varphi(n-k-1)+1}$. Moreover, both these multiples will occur  in this range only when $n-k = 1,2,4,8$. Now the proof of the theorem follows if $2(n-k) \leq n-1$. If $n-1 <2(n-k)$ then the injectivity of the map $f^*$ gives $H^j(V_{n,k}) = 0$, and hence $w_j(\xi) = 0$, for $n-1 < j \leq 2(n-k)$. This completes the proof.
\end{proof}

If we assume  $n \geq 2k$ then $n-1 < 2(n-k)$. Then the proof of the following corollary follows from Theorem \ref{tanaka}.
	
\begin{Cor}\label{stunted}
	Let $V_k(\br^n)$ be a Stiefel manifold with $n \geq 2k$. Then $w_i(\xi) = 0$ for $i \leq n-1$ and $i \neq 2^{\varphi(n-k-1)}$ for any vector bundle $\xi$ over $V_k(\br^n)$. \qed	
\end{Cor}

If we fix $k$ and vary $n$ then we have the following corollary. 
 \begin{Cor}
 Let $k$ be fixed. Then except for finitely many values of $n$, the Stiefel-Whitney classes $w_i(\xi)=0$ for  $i\leq n-1$ and any vector bundle $\xi$ over $V_k(\br^n)$.
 \end{Cor}
 \begin{proof}
 The proof follows from Corollary \ref{stunted} by using the fact that $n-1<2^{\varphi(n-k-1)}$ except for finitely many values of $n.$
  \end{proof} 

%
%
%
%
%
  
  In view of Theorem \ref{tanaka}, it will be interesting to know whether there exists a vector bundle $\xi$ over $V_k(\br^n)$ such that $w_{2^{\varphi(n-k-1)}}(\xi) \neq 0$. We have complete answer when $2^{\varphi(n-k-1)} = n-k$. We observed in the above proof that $2^{\varphi(n-k-1)}=n-k$ if and only if $n-k=1,2,4,8.$ In the case when $n-k = 1,2$, the existence of a vector bundle $\xi$ such that $w_{n-k}(\xi) \neq 0$ is a consequence of the fact that $H^1(V_k(\br^{k+1});\bz_2) \neq 0$ and the mod-2 reduction map $H^2(V_k(\br^{k+2}); \bz) \rightarrow H^2(V_k(\br^{k+2});\bz_2)$ is the projection map $\bz \rightarrow \bz_2$ (cf. \cite{knt}). In the following example, when $n-k=4,8$ we construct a vector bundle $\xi$ over $V_k(\br^n)$ such that $w_{n-k}(\xi)\neq 0.$
  
 \begin{Exm}\label{exmp} \em{
  Let $\alpha: Spin(n) \rightarrow V_k(\br^n)$ be the principal $Spin(n-k)$-bundle over $V_k(\br^n)=Spin(n)/Spin(n-k)$. If $\wt{RO}Spin(n-k)$ and $\wt{R}Spin(n-k)$ are the reduced real and complex representation rings respectively, then we have the following commutative diagram: 
  
  \begin{equation}\label{diag1}
  \xymatrix{
  \wt{RO}Spin(n-k) \ar[r]\ar[d] & \wt{KO}(Spin(n)/Spin(n-k)) \ar[ld]^{f^*}  \\
   \wt{KO}(Spin(n-k+1)/Spin(n-k)) &
}
  \end{equation}  
 Here $f: S^{n-k} = Spin(n-k+1)/Spin(n-k) \rightarrow Spin(n)/Spin(n-k)$ is the natural inclusion.  
  In the case when $n-k = 8$, the map $\wt{RO}Spin(8) \rightarrow \wt{KO}(S^8)$ in Diagram \ref{diag1} 
is surjective (cf. p.195, \cite{husemoller}) and hence the map $f^*$ is surjective If $[\xi] \in \wt{KO}(\mathbb{S}^8)$ is the class of the Hopf bundle over $S^8$ then there exists a bundle $\eta$ over $V_k(\br^n)$ such that $f^*([\eta]) = [\xi]$. As $w_8(\xi) \neq 0$, we have $w_8(\eta) \neq 0$.

 Next when $n-k =4$, we use the following diagram: \begin{equation} \label{diag2}
  \xymatrix{
  \wt{R}Spin(n-k) \ar[r]\ar[d] & \wt{K}(Spin(n)/Spin(n-k)) \ar[ld]^{f^*}  \\
   \wt{K}(Spin(n-k+1)/Spin(n-k)) &
}
  \end{equation} 
The map $\wt{R}Spin(4) \rightarrow \wt{K}(S^4)$ in Diagram \ref{diag2}
is surjective (cf. p.195, \cite{husemoller}). Using the fact that the Hopf bundle $\xi$ over $S^4$ is a complex vector bundle with $w_4(\xi) \neq 0$, we proceed as above to conclude that there exists a complex vector bundle $\eta$ over $V_k(\br^n)$ such that the Stiefel-Whitney class $w_4(\eta_{\br})$ of the underlying real bundle $\eta_{\br}$ is nonzero.

}

\end{Exm}

\section{Proof of Theorem \ref{main}}
Recall the description of the cohomology ring $H^*(V_k(\br^n))$ as in the Introduction. Because of the relations among the generators $a_{n-k}, a_{n-k+1}, \cdots, a_{n-1}$, we can write any nonzero cohomology class $x \in H^j(V_{k}(\br^n))$ as $$x = \sum a_{i_1}\cdot a_{i_2}\cdots a_{i_r}$$ such that $i_{t} < i_{t+1}$. If a monomial $a_{i_1}\cdot a_{i_2}\cdots a_{i_r}$ in the above summand represents a nonzero cohomology class then  we have $$(n-k) +(t-1) \leq \deg a_{i_t} \leq n-1 -r +t.$$  This implies that
$$\sum_{t =1}^{r} (n-k) +(t-1) \leq \sum_{t=1}^{r} a_{i_t} \leq \sum_{t=1}^r n -1 -r +t.$$ Hence,
$r(n-k) + r(r-1)/2 \leq j \leq r(n-1) - r(r-1)/2$. 

For $0\leq p\leq k$, we define $T_p$ as  the set $\{j\in \mathbb{N}: p(n-k) + p(p-1)/2 \leq j\leq p(n-1) - p(p-1)/2\}$. Therefore, by the above discussion we have the following lemma. 

\begin{Lem}\label{monomial}If $x = a_{i_1}\cdot a_{i_2}\cdots a_{i_r}$ with $i_{t} < i_{t+1}$ represents a nonzero cohomology class of $V_k(\br^n)$, then $\deg x  \in T_r$. \qed
\end{Lem}

If we assume $n > k(k+4)/4$, then in the following lemma we give an upper bound for the length of each $T_p$.
\begin{Lem}\label{difference}
Let $n > k(k+4)/4$. Then $|r_1 -r_2| < n-k$ for any $p$ and  $r_1, r_2 \in T_p$.

\end{Lem}
\begin{proof}
For any $r_1,r_2\in T_p$, we have $|r_1-r_2|\leq p(n-1)-p(p-1)/2-p(n-k)-p(p-1)/2=p(k-p).$ The maximum value of the set $\{p(k-p):1\leq p\leq k\}$ is $k^2/4$ if $k$ is even and $(k^2-1)/4$ if $k$ is odd. Since $n>k(k+4)/4$ if and only if $n-k>k^2/4,$ we have $|r_1-r_2|<n-k.$
\end{proof}

In the following lemma, we derive some results involving binomial coefficients which we shall use in the proof of Theorem \ref{main}.

\begin{Lem}\label{binomial}
Let $s$ be an odd number and  $r \leq 2^t$. Then the binomial coefficients	
\begin{enumerate}
 \item ${2^ts+r-1\choose r}$
is even if and only if $r\neq 0,2^{t}$. 
\item[]
\item ${2^ts-1\choose 2^{t+1}}$ is odd if $s \equiv 3 \pmod 4$.
\end{enumerate}
\end{Lem}
\begin{proof}
To prove Statement (1), we note that if $r \neq 0$ then \scriptsize $${2^{t}s+r-1\choose r} =  \left(\frac{2^ts }{r}\right)\left(  \prod_{l=1}^{[(r-1)/2]}\frac{2^ts +2}{2l}\right) \left(  \prod_{l=1}^{[r/2]}\frac{2^ts +2l-1}{2l-1}\right)$$\normalsize
Now it is easy to see that the second and third products in the right hand side of the above equality can be written as ratios of two odd integers. Further, $(2^ts/r)$ can be written as a ratio of two odd integers if and only if $r = 2^t$. From here we conclude Statement (1).

  Next we prove Statement (2). We first note that
\scriptsize	$${2^ts-1\choose2^{t+1}}= \left(\prod_{l=1}^{2^t}\frac{2^ts-2l}{2l} \right) \left(\prod_{l=1}^{2^t}\frac{2^ts-2l -1}{2l-1} \right)$$ \normalsize
Now  if $ l \neq 2^{t-1} \mbox{ or } 2^t$, then $\frac{2^ts-2l}{2l}$ can be written as a ratio of two odd integers. On the other hand if $l = 2^{t-1} \mbox { and } 2^t$ then the product \scriptsize
$$\left(\frac{2^ts-2^t}{2^t}\right)\left(\frac{2^ts-2^{t+1}}{2^{t+1}}\right)=(s-1)(s-2)/2,$$
\normalsize
 which is an odd number as $s \equiv 3 \pmod 4$. This completes the proof of Statement (2).   \end{proof}

%

%
%

 We now prove Theorem \ref{main}.
	 \begin{proof}[\textbf{Proof of Theorem \ref{main}}]
	Let $i = 2^{q+t_1}+2^{q+t_2}+ \cdots+2^{q+t_m}$ with $t_j\geq0$ and $t_j<t_{j+1}$. If $i$ is a power of 2 (i.e.,  when $m =1$) or $H^{i}(V_k(\br^n))=0$, then the first statement of the theorem follows easily. Next we assume that $m>1$ and $H^{2^qr}(V_k(\br^n))\neq0$. By Wu's formula we get
	 \scriptsize	$$\begin{array}{ccl}
	 	Sq^{2^{q+t_1}}(w_{i-2^{q+t_1}}(\xi))& = &\displaystyle\sum_{r=0}^{2^{q+t_1}}{i-2^{q+t_1+1}+r-1\choose r}w_{2^{q+t_1}-r}(\xi)\cdot w_{i-2^{q+t_1}+r}(\xi)\\\\
	 	& =& w_{2^{q+t_1}}(\xi)\cdot w_{i-2^{q+t_1}}(\xi)+w_{i}(\xi).
	 	\end{array}$$ \normalsize 
	 	The last equality above follows by Lemma \ref{binomial}(1). 	
	 	Next we prove that the left hand side of the above equation is zero. For this it is enough to prove that if $x = a_{i_1} \cdot a_{i_2}  \cdots a_{i_{p}}$, with $i_j < i_{j+1}$,  is a nonzero cohomology class of degree $i-2^{q+t_1}$, then the Steenrod square $Sq^{2^{q+t_1}}(x)=0$. For this first note that $$ Sq^{2^{q+t_1}}(x) =Sq^{2^{q+t_1}}(a_{i_1} \cdot a_{i_2}  \cdots a_{i_{p}})=\sum_{l_1 + \cdots + l_{p} = 2^{q+t_1}}  Sq^{l_1}(a_{i_1}) \cdots Sq^{l_{p}}(a_{i_{p}}).$$  We shall show that each summand in the right hand side of the above equation is zero.
	 	As the monomial $a_{i_1} \cdot a_{i_2}  \cdots a_{i_{p}}$ represents a nonzero cohomology class, it follows by Lemma \ref{monomial} that its degree, $i-2^{q+t_1} \in T_{p}$.  If a summand $ Sq^{l_1}(a_{i_1}) \cdots Sq^{l_{p}}(a_{i_{p}})$ is nonzero then 
	 	for all $j$  we have $l_j + i_j \leq n-1$, $Sq^{l_j}(a_{i_j}) = a_{{i_j} + l_j}$. Moreover, as $n\geq 2k$, we have $a_{i_j}^2 =0$ for all $j$ and this will imply that $l_{j_1} +i_{j_1} \neq l_{j_2} +i_{j_2}$ for $j_1 \neq j_2$. Hence, $$
	 	p(n-k) +p(p-1)/2 \leq \sum_{j=1}^{p} i_j + l_j = i \leq p(n-1) - p(p-1)/2.$$
	 	This implies that $i \in T_{p}$. Since, $i-2^{q+t_1}$ also belongs to $T_p$, the difference, $i - (i-2^{q+t_1})= 2^{q+t_1}  \geq 2^q \geq n-k$, gives a contradiction to Lemma \ref{difference} and hence, we conclude that  $Sq^{2^{q+t_1}}(x)=0$.  This proves that $w_{i}(\xi)= w_{2^{q+t_1}}(\xi)\cdot w_{i-2^{q+t_1}}(\xi).$ Now the proof of  the first statement follows by induction on $m$.
	 	
	 	Now we prove the last statement of the theorem by applying induction on the set $\{i:i \text{ is not a multiple of }2^q\}$. If $i<2^q$, then $w_i(\xi)=0$ by hypothesis.  Next assume that $i > 2^q$, $H^i(V_k(\br^n) \neq 0$ and $i$ is not a multiple of $2^q$. We can write $i$ as $i=2^ts$ where $s$ is odd, $s\geq3$ and $t<q$. Applying Lemma \ref{binomial}(1) on Wu's formula we get 
	 	$$Sq^{2^t}(w_{2^t(s-1)}(\xi))=w_i(\xi).$$  
	 	If $2^t(s-1)$ is not a multiple of $2^q$ or $H^{2^t(s-1)}(V_k(\br^n))=0$, then by induction we have $w_i(\xi)=0.$  Now assume that $H^{2^t(s-1)}(V_k(\br^n))\neq0$ and $2^t(s-1)$ is a multiple of $2^q$. Let $2^t(s-1) = 2^{q+t_1} + 2^{q+t_2} + \cdots + 2^{q+t_m}$ with $t_j < t_{j+1}$. We have the following two cases:
	 	
	 	\textbf{Case $m>1$:}	 	
As, $2^t(s-1)$ is a multiple of $2^q$, by the first statement of the theorem, we have
	 	\begin{equation}\label{eqn2}	 	
	\begin{array}{ccl}w_{i}(\xi) &= &Sq^{2^t}(w_{2^{q+t_1}}(\xi)\cdot w_{2^{q+t_2}}(\xi)\cdots w_{2^{q+t_m}}(\xi))\\
	 	 & &\\	 	
	 	&=& \underset{l_1+\cdots l_m=2^t}{\sum}Sq^{l_1}(w_{2^{q+t_1}}(\xi))\cdots Sq^{l_m}(w_{2^{q+t_m}}(\xi)).\end{array}
	 		\end{equation}		
	 	Now observe that for each $j$, we have $l_j \leq 2^{t} < 2^q$, and hence, the Steenrod square, $$Sq^{l_j}(w_{2^{q+t_j}}(\xi))={2^{q+t_j}-1\choose l_j}w_{2^{q+t_j}+l_j}(\xi).$$ As $m>1$, we have $2^{q+t_j}< 2^t(s-1)$. Therefore $2^{q+t_j}+l_j<2^ts$. Further if for some $j$, we have $l_j>0$, then $2^{q+t_j}+l_j$ is not a multiple of $2^q$ and hence by induction, $w_{2^{q+t_j}+l_j}(\xi)=0.$ Now observe that in each summand in the right hand side of Equation \ref{eqn2}, there is at least one $l_j$ such that $l_j>0.$ Therefore, $$w_{2^ts}(\xi)=\sum_{l_1+\cdots l_m=2^t}Sq^{l_1}(w_{2^{q+t_1}}(\xi))\cdots Sq^{l_m}(w_{2^{q+t_m}}(\xi))= 0.$$
	 	
	 \textbf{Case $m=1$:} In this case, $2^t(s-1) = 2^{q+t_1}$ for some $t_1 \geq 0$. Thus $s-1$ is a power of $2.$ First we consider the case when $s\geq 5.$ Here we observe that $2^{t+1} < 2^ts - 2^{t+1}.$ Since, $s-1 \equiv 0 \pmod 4$, then by Lemma \ref{binomial}(2), we have $$	Sq^{2^{t+1}}(w_{2^ts-2^{t+1}}(\xi))=w_{2^{t+1}}(\xi)\cdot w_{2^ts-2^{t+1}}(\xi)+w_{2^ts}(\xi).$$
	 	As $2^ts-2^{t+1}=2^t(s-1) -2^t = 2^{q+t_1} -2^t = 2^t(2^{q-t+t_1}-1)$, we have that $2^ts-2^{t+1}$ is not a multiple of $2^q$ and hence, by induction, $w_{2^ts-2^{t+1}}(\xi) = 0$.
	 	Therefore, $w_{2^ts}(\xi) = 0$. Next we deal with the case $s = 3$. We observe that in this case, $t = q-1$ and $t_1 =0$. Thus, we obtain
	 	$$Sq^{2^{q-1}}(w_{2^q}(\xi))=w_{2^{q-1}3}(\xi) = w_i(\xi).$$
	 If $2^q \in T_1$, then $w_{2^q}(\xi) = a_j$ for some $j$, as $n \geq 2k$. Therefore $w_i(\xi) = Sq^{2^{q-1}}(a_j)$. So $w_i(\xi) = 0$ if $i > n-1$. If $i \leq n-1$, then $w_i(\xi)$ is again zero, as $w_{2^q}(\xi)$ is the only nonzero Stiefel-Whitney class up to degree $n-1$ (cf. Corollary \ref{stunted}). Next we assume that $2^q \not \in T_1$. Since $w_{2^q}(\xi) \neq 0$, we have that $2^q \in T_p$ for some $p$ such that $p \geq 2$.	To prove $w_{2^{q-1}3}(\xi)=0$,  we consider the  Steenrod square operation,
	 	\begin{equation}\label{eqn}
	 	Sq^{2^{q-1}}(a_{i_1} \cdot a_{i_2}  \cdots a_{i_{p}})=\sum_{l_1 + \cdots + l_{p} = 2^{q-1}}  Sq^{l_1}(a_{i_1}) \cdots Sq^{l_{p}}(a_{i_{p}})
	 	\end{equation}
	 	 on a monomial $a_{i_1} \cdot a_{i_2}  \cdots a_{i_{p}}$ such that $i_j<i_{j+1}$, which represents a nonzero cohomology class of degree $2^q$. By Lemma \ref{monomial}, the degree of the monomial, $2^q \in T_p$. If a summand $ Sq^{l_1}(a_{i_1}) \cdots Sq^{l_{p}}(a_{i_{p}})$ is nonzero then 
	 	for all $j$  we have $l_j + i_j \leq n-1$, $Sq^{l_j}(a_{i_j}) = a_{{i_j} + l_j}.$ Moreover, as $n\geq 2k$, we have $l_{j_1} +i_{j_1} \neq l_{j_2} +i_{j_2}$ for $j_1 \neq j_2$. Hence, $$
	 	p(n-k) +p(p-1)/2 \leq \sum_{j=1}^{p} i_j + l_j = 2^{q-1}3 \leq p(n-1) -p(p-1)/2.$$
	 	This implies that $2^{q-1}3 \in T_{p}$. As $p \geq 2$, we have $2^q \geq 2(n-k)+1$ and 
	 	 this implies that $2^{q-1} > n-k$. Therefore, the difference $2^{q-1}3 - 2^{q}   = 2^{q-1}>n-k$. This is a contradiction to Lemma \ref{difference}.  This shows that each summand in the right hand side of Equation \ref{eqn} is zero. Hence $w_{2^{q-1}3}(\xi) =0$ if $p \geq 2$. 
	 	
	 	%
	 	%
	 	%
	 	%

	 	%
	 	%
	 \end{proof}

\subsection*{Acknowledgment:} The authors thank Aniruddha C. Naolekar and Parameswaran Sankaran for their valuable suggestions and comments.

\end{document}